\documentclass[12pt,twosides]{amsart}
\usepackage{amssymb,amsmath,amsthm, amscd, enumerate, mathrsfs}
\usepackage{graphicx, hhline}
\usepackage[all]{xy}
\usepackage{braket}
\usepackage[usenames]{color}
\usepackage{hyperref}
\usepackage{fancyhdr}
\usepackage[top=30truemm,bottom=30truemm,left=30truemm,right=30truemm]{geometry}
\usepackage{varioref}

\hypersetup{colorlinks=true}

\title{A note on lc-trivial fibrations}
\author{Kenta Hashizume}
\date{2023/10/09, version 0.06}
\address{Department of Mathematics, Faculty of Science, Niigata University, Niigata 950-2181, Japan}
\address{Institute for Research Administration, Niigata University, Niigata 950-2181, Japan}
\email{hkenta@math.sc.niigata-u.ac.jp}

\DeclareMathOperator{\mult}{mult}
\DeclareMathOperator{\Supp}{Supp}
\newtheorem{thm}{Theorem}[section]

\newtheorem{lem}[thm]{Lemma}

\theoremstyle{definition}
\newtheorem{defn}[thm]{Definition}

\newtheorem*{note*}{Notation}

\newtheorem*{ack}{Acknowledgments}

\newtheorem{step2}{Step}

\newtheorem*{claim*}{Claim}
\begin{document}

\maketitle

\begin{abstract}
For every lc-trivial fibration $(X,\Delta) \to Z$ from an lc pair, we prove that after a base change, there exists a positive integer $n$, depending only on the dimension of $X$, the Cartier index of $K_{X}+\Delta$, and the sufficiently general fibers of $X \to Z$, such that $n(K_{X}+\Delta)$ is linearly equivalent to the pullback of a Cartier divisor. 
\end{abstract}

\tableofcontents 

\section{Introduction}
Throughout this note, we will work over the complex number field. 

An lc-trivial fibration is a fibration together with an lc pair such that the log canonical divisor of the lc pair is relatively $\mathbb{Q}$-linearly trivial, i.e., $\mathbb{Q}$-linearly equivalent to the pullback of a $\mathbb{Q}$-Cartier divisor on the base variety. 
For the definition of lc-trivial fibrations, see Definition \ref{defn--lctrivial-fib}. 
This kind of fibration naturally appears and plays an important role in birational geometry. 
For example, every fibration induced by the log canonical divisor of a good minimal model is an lc-trivial fibration, every fibration from a Calabi--Yau manifold is also an lc-trivial fibration, and we can define the structure of an lc-trivial fibration for any Mori fiber space.  
Because of the relative $\mathbb{Q}$-linear triviality of the log canonical divisors of lc pairs, it has been revealed that the geometries of the total variety, the general fiber, and the base variety of an lc-trivial fibration are closely related to each other (\cite{kawamata-subad-II}, \cite{ambro1}, \cite{ambro}, \cite{fg-bundle}, \cite{fg-lctrivial}, \cite{birkar-fib}, \cite{birkar-logCY}, \cite{birkar-nefpart}). 
Especially, a relation between the total variety and the base variety, called the canonical bundle formula and developed by Kawamata (\cite{kawamata-subad-II}) and Ambro (\cite{ambro1}, \cite{ambro}), plays a fundamental role in the recent development of birational geometry (\cite{fg-bundle}, \cite{birkar-fib}, \cite{jiao-CY-fib}). 
For an lc-trivial fibration $\pi \colon (X,\Delta) \to Z$ and a $\mathbb{Q}$-Cartier divisor $D$ on $Z$ such that $K_{X}+\Delta \sim_{\mathbb{Q}}\pi^{*}D$, the canonical bundle formula (Definition \ref{defn--canbundleformula}) asserts that $D$ can be written as $K_{Z}+\boldsymbol{\rm B}_{Z}+\boldsymbol{\rm M}_{Z}$ the sum of the canonical divisor $K_{Z}$ of $Z$, the discriminant part $\boldsymbol{\rm B}_{Z}$, which contains the information of singular fibers of $\pi$, and the moduli part $\boldsymbol{\rm M}_{Z}$, which is concerned with $\pi$ as a space parametrizing the fibers. 
The discriminant part and the moduli part form a generalized pair $(Z,\boldsymbol{\rm B}_{Z}+\boldsymbol{\rm M})$, which was introduced by Birkar--Zhang \cite{bz}. 
Currently, the study of connection between the generalized pair $(Z,\boldsymbol{\rm B}_{Z}+\boldsymbol{\rm M})$ and the lc pair $(X,\Delta)$ is a central topic of lc-trivial fibration. 

In this note we discuss a relation between the Cartier index of $K_{X}+\Delta$ and that of $K_{Z}+\boldsymbol{\rm B}_{Z}+\boldsymbol{\rm M}_{Z}$ after a base change.

\begin{thm}[Main result]\label{thm--lctrivialfibration-boundedfiber}
For every $d$, $m \in \mathbb{Z}_{>0}$ and $v \in \mathbb{R}_{>0}$, there exists $n \in \mathbb{Z}_{>0}$, depending only on $d$, $m$, and $v$, satisfying the following. 
Let $(X,\Delta)$ be a projective lc pair, let $\pi \colon (X,\Delta) \to Z$ be an lc-trivial fibration with the sufficiently general fiber $F$, and let $A\geq 0$ be a Weil divisor on $X$ such that 
\begin{itemize}
\item
$\dim X=d$, 
\item
$mf^{*}(K_{X}+\Delta)$ is Cartier for some resolution $f\colon X' \to X$ of $X$, 
\item
$A|_{\pi^{-1}(U)}$ is $\mathbb{Q}$-Cartier and ample over $U$ for some open subset $U \subset Z$,
\item
$(F,\Delta|_{F}+tA|_{F})$ is an lc pair for some real number $t>0$, and 
\item
${\rm vol}(A|_{F}) = v$.  
\end{itemize}
Then there exists a generalized lc pair $(Z,\boldsymbol{\rm B}_{Z}+\boldsymbol{\rm M})$ defined with the canonical bundle formula such that
\begin{itemize}
\item
 $n(K_{X}+\Delta)\sim n\pi^{*}(K_{Z}+\boldsymbol{\rm B}_{Z}+\boldsymbol{\rm M}_{Z})$, 
\item
$n \boldsymbol{\rm M}$ is b-Cartier, and
\item
$n\phi^{*}(K_{Z}+\boldsymbol{\rm B}_{Z}+\boldsymbol{\rm M}_{Z})$ is Cartier for some resolution $\phi \colon Z' \to Z$ of $Z$.   
\end{itemize}
Furthermore, if there is a klt pair structure on $Z$ then $n(K_{Z}+\boldsymbol{\rm B}_{Z}+\boldsymbol{\rm M}_{Z})$ is Cartier. 
\end{thm}

For the definitions of b-divisor and b-Cartier property, see Definition \ref{defn--b-divisor}. 
Theorem \ref{thm--lctrivialfibration-boundedfiber} looks similar to \cite[Lemma 7.4]{birkar-nefpart} by Birkar. 
In both statements, we fix the dimension of the lc pair and the set of the coefficients of the boundary divisor, and we assume a kind of boundedness condition on the general fibers of the lc-trivial fibration. 
In Theorem \ref{thm--lctrivialfibration-boundedfiber}, the linear equivalence between  $n(K_{X}+\Delta)$ and $n\pi^{*}(K_{Z}+\boldsymbol{\rm B}_{Z}+\boldsymbol{\rm M}_{Z})$ and the b-Cartier property of $n\boldsymbol{\rm M}$ are not new because the two properties have already been proved in \cite[Lemma 7.4]{birkar-nefpart}. 
The new contribution of Theorem \ref{thm--lctrivialfibration-boundedfiber} is that the b-Cartier property of $n(K_{Z}+\boldsymbol{\rm B}_{Z}+\boldsymbol{\rm M}_{Z})$ as a b-divisor depends on the b-Cartier index of $K_{X}+\Delta$ as a b-divisor. 
Birkar's result \cite[Lemma 7.4]{birkar-nefpart} plays a crucial role for the boundedness of the base varieties of Iitaka fibrations (see \cite[Theorem 1.3]{birkar-hacon} by Birkar--Hacon), whereas Theorem \ref{thm--lctrivialfibration-boundedfiber} is useful to study effective base point freeness.

\begin{thm}\label{thm--eff-base-point-free}
For every $d$, $m \in \mathbb{Z}_{>0}$, and $v \in \mathbb{R}_{>0}$, there exists $n \in \mathbb{Z}_{>0}$, depending only on $d$, $m$, and $v$, satisfying the following.  
Let $(X,\Delta)$ be a projective klt pair such that $e(K_{X}+\Delta)$ is semi-ample for an $e\in \{1,-1\}$, let $\pi \colon X \to Z$ be the contraction induced by $e(K_{X}+\Delta)$, and let $A$ be a $\mathbb{Q}$-Cartier Weil divisor on $X$ such that 
\begin{itemize}
\item
$\dim X=d$, 
\item
$m(K_{X}+\Delta)$ is Cartier, and
\item
${\rm vol}(A|_{F}) = v$, where $F$ is a sufficiently general fiber of $\pi$. 
\end{itemize}
Then $ne(K_{X}+\Delta)$ is base point free and the linear system $|ne(K_{X}+\Delta)|$ defines $\pi$. 
\end{thm}

We emphasize that we do not assume the divisor $A$ in Theorem \ref{thm--eff-base-point-free} to be effective. 
In the case where $K_{X}+\Delta$ is big in Theorem \ref{thm--eff-base-point-free},  the statement is a spacial case of Koll\'ar's effective base point free theorem \cite{kollar-eff-basepoint-free}. 
Theorem \ref{thm--eff-base-point-free} is a partial generalization of \cite[Theorem 1.1]{kollar-eff-basepoint-free} to the case where the log canonical divisor or anti log canonical divisor of a klt pair has intermidiate Kodaira dimension. 
After this note was announced, Masamura \cite{masamura} gave an example of minimal surface for which the effective base point freeness does not hold. 
This shows that the existence of the divisor $A$ in Theorem \ref{thm--eff-base-point-free} can not be removed unconditionally.

By using the main result (Theorem \ref{thm--lctrivialfibration-boundedfiber}), we can prove another type of effective base point free theorem in the non-klt case.  
In Section \ref{sec4}, we study lc-trivial fibrations whose moduli parts satisfy the log bigness on a certain higher birational model. 
The motivation of the topic comes from the following result, which is a consequence of the argument in \cite[Proof of Proposition 4.4]{floris-lazic} by Floris--Lazi\'c (see also \cite[Theorem 1.2]{hu-lctrivial-b-div} by Hu). 

\begin{thm}[{\cite{floris-lazic}}]\label{thm--florislazic}
\label{thm--can-bundle-formula-modulipart}
Let $(X,\Delta) \to Z$ be an lc-trivial fibration. 
Then there exists a log smooth Ambro model $Z' \to Z$ such that $\boldsymbol{\rm M}_{Z'}$ is log abundant with respect to $(Z', \boldsymbol{\rm B}_{Z'})$, i.e., for every stratum $T'$ of $(Z',\boldsymbol{\rm B}_{Z'})$, there exist a birational morphism $h \colon W \to T'$ from a normal projective variety $W$, a surjective morphism $\psi \colon W \to V$ to a normal projective variety $V$, and a nef and big $\mathbb{Q}$-divisor $N$ on $V$ such that $h^{*}(\boldsymbol{\rm M}_{Z'}|_{T'})\sim_{\mathbb{Q}} \psi^{*}N$. 
\end{thm}

For the definition of log smooth Ambro model, see Definition \ref{defn--ambromodel}. 
Note that $Z'$ in Theorem \ref{thm--florislazic} is not unique. 
Lc-trivial fibrations with log big moduli parts are defined by the log bigness of the moduli part on a log smooth Ambro model (Definition \ref{defn--lc-trivial-log-big}). 
The log bigness is one of spacial cases in the property of being log abundant. 
Hence, lc-trivial fibrations with log big moduli parts appear as special lc-trivial fibrations. 
By definition, an lc-trivial fibration is an lc-trivial fibration with log big moduli part when there is a log smooth Ambro model $Z' \to Z$ such that the map $\psi\circ h^{-1}\colon T' \dashrightarrow V$ in Theorem \ref{thm--can-bundle-formula-modulipart} is birational for every stratum $T'$ of $(Z',\boldsymbol{\rm B}_{Z'})$. 

By combining Theorem \ref{thm--lctrivialfibration-boundedfiber} with a result from \cite{has-iitakafibration}, we will prove an effective base point free theorem for lc-trivial fibrations with log big moduli parts. 

\begin{thm}[Theorem \ref{thm--lctrivialfibration-boundedfiber-abundance} and Theorem \ref{thm--lctrivialfibration-boundedfiber-antiabundance}]
For every $d$, $m \in \mathbb{Z}_{>0}$, and $v \in \mathbb{R}_{>0}$, there exists $n \in \mathbb{Z}_{>0}$, depending only on $d$, $m$, and $v$, satisfying the following.  
Let $\pi \colon (X,\Delta) \to Z$ be an lc-trivial fibration with log big moduli part, let $F$ be the sufficiently general fiber of $\pi$, and let $A \geq 0$ be a Weil divisor on $X$ such that 
\begin{itemize}
\item
$\dim X=d$, 
\item
$me(K_{X}+\Delta)$ is nef and Cartier for an $e \in \{1,-1\}$, 
\item
$A|_{\pi^{-1}(U)}$ is $\mathbb{Q}$-Cartier and ample over $U$ for some open subset $U \subset Z$,
\item
$(F,\Delta|_{F}+tA|_{F})$ is an lc pair for some real number $t>0$, and 
\item
${\rm vol}(A|_{F}) = v$. 
\end{itemize}
Then $ne(K_{X}+\Delta)$ is base point free. 
\end{thm}

Moreover, good minimal models or Mori fiber spaces always exist for projective lc pairs admitting the structure of an lc-trivial fibration with log big moduli part. 
\begin{thm}[=Theorem \ref{thm--minimalmodeltheory-logbigmoduli}]
Let $(X,\Delta)$ be a projective lc pair admitting an lc-trivial fibration with log big moduli part. 
Then $(X,\Delta)$ has a good minimal model or a Mori fiber space. 
\end{thm}

Although the theorem is not an application of Theorem \ref{thm--lctrivialfibration-boundedfiber}, we hope that the theorem is independent of interest.

\begin{ack}
The author was partially supported by JSPS KAKENHI Grant Number JP22K13887. 
The author is grateful to Doctor Masafumi Hattori for fruitful discussions which improved Theorem \ref{thm--eff-base-point-free} in the first draft. 
The author thanks Professor Osamu Fujino for informing him of the paper \cite{fujino-eff-slc}. 
The author thanks the referee for comments and suggestions. 
\end{ack}

\section{Definitions}

In this section, we collect definitions used in this note. 

Let $a$ be a real number. 
Then $\lfloor a \rfloor$ denotes the integer satisfying $a-1 < \lfloor a \rfloor \leq a$.
Let $D$ be an $\mathbb{R}$-divisor on a variety and let $D=\sum_{i}d_{i}D_{i}$ be the decomposition of $D$ into prime divisors. 
Then $\lfloor D \rfloor$ is defined to be $\sum_{i}\lfloor d_{i}\rfloor D_{i}$. 

A {\em contraction} is a projective morphism $f \colon X \to Z$ of varieties such that $f_{*}\mathcal{O}_{X} \simeq \mathcal{O}_{Z}$. 

Let $X$ be a normal variety and $D$ an $\mathbb{R}$-divisor on $X$. 
Then $(X,D)$ is {\em log smooth} if $X$ is smooth and $D$ has simple normal crossing support. 
A {\em log resolution} of $(X,\Supp D)$ is a projective birational morphism $f\colon Y\to X$ from a smooth variety $Y$ such that the exceptional locus ${\rm Ex}(f)$ of $f$ is pure codimension one and ${\rm Ex}(f)\cup {\rm Supp}\,f_{*}^{-1}D$ has simple normal crossing support.

\begin{defn}[b-divisor]\label{defn--b-divisor}
Let $X$ be a normal variety. 
Consider all proper birational morphisms $f \colon Y \to X$ from normal varieties $Y$ and the sets ${\rm WDiv}_{\mathbb{Q}}(Y)$ of $\mathbb{Q}$-divisors on $Y$. 
For two proper birational morphisms $Y_{1} \to X$ and $Y_{2} \to X$ such that the induced birational map $\phi \colon Y_{1} \dashrightarrow Y_{2}$ is a morphism, we define the map
$$\phi_{*} \colon {\rm WDiv}_{\mathbb{Q}}(Y_{1}) \longrightarrow {\rm WDiv}_{\mathbb{Q}}(Y_{2})$$
by taking the birational transform. 
Then a {\em $\mathbb{Q}$-b-divisor} $\boldsymbol{\rm D}$ on $X$ is an element of the inverse limit
$$\boldsymbol{\rm WDiv}_{\mathbb{Q}}(X):=\underset{\longleftarrow}{\rm lim}{\rm WDiv}_{\mathbb{Q}}(Y)$$
defined with all such $\phi_{*}$ as above. 

For a $\mathbb{Q}$-b-divisor $\boldsymbol{\rm D}$ on $X$ and a normal variety $Y$ with a proper birational morphism $Y \to X$, the image of $\boldsymbol{\rm D}$ by the natural projection $\boldsymbol{\rm WDiv}_{\mathbb{Q}}(X) \to {\rm WDiv}_{\mathbb{Q}}(Y)$ is called the {\em trace of $\boldsymbol{\rm D}$ on $Y$}, and it is denoted by $\boldsymbol{\rm D}_{Y}$. 

Let $\boldsymbol{\rm D}$ be a $\mathbb{Q}$-b-divisor on $X$. 
We say that {\em $\boldsymbol{\rm D}$ descends to a normal variety $Y$} if there exist a proper birational morphism $Y \to X$ from $Y$ and a $\mathbb{Q}$-Cartier divisor $D$ on $Y$ such that for any normal variety $Y'$ together with a proper birational morphism $Y' \to X$, by taking a common resolution $g \colon V \to Y$ and $g' \colon V \to Y'$ of $Y \dashrightarrow Y'$ we have $\boldsymbol{\rm D}_{Y'}=g'_{*}g^{*}D$. 
In this situation, we also say that $\boldsymbol{\rm D}$ is the {\em Cartier closure of $D$}, and we use the notation $\overline{D}$ to denote $\boldsymbol{\rm D}$. 
If, additionally, $D$ is Cartier, then we say that $\boldsymbol{\rm D}$ is {\em b-Cartier}. 
\end{defn}

We freely use the notations and the definitions in \cite{kollar-mori} and \cite{bchm} for singularities of pairs except that $a(D,X,\Delta)$ denotes the {\em log discrepancy} of $D$ with respect to a pair $(X,\Delta)$ in this note. 
We use the definition of generalized pairs in \cite{bz} and \cite{birkar-compl}. 
Note that the notation of generalized pair in this note is different from that in \cite{bz} or \cite{birkar-compl}.
In this note, we only deal with generalized pairs in the framework of $\mathbb{Q}$-divisors. 

\begin{defn}[Generalized lc pair]\label{defn--gen-pair}
A {\em generalized pair} $(X,B+\boldsymbol{\rm M})/Z$ consists of 
\begin{itemize}
\item
a projective morphism $X\to Z$ from a normal variety $X$ to a variety $Z$, 
\item
an effective $\mathbb{Q}$-divisor $B$ on $X$, and
\item
a $\mathbb{Q}$-b-divisor $\boldsymbol{\rm M}$ on $X$ that is the Cartier closure of a $\mathbb{Q}$-Cartier divisor $M'$, which is nef over $Z$, on a normal variety $X'$ with a projective birational morphism $X' \to X$  
\end{itemize}
such that $K_{X}+B+\boldsymbol{\rm M}_{X}$ is $\mathbb{Q}$-Cartier. 
When $Z$ is a point, we simply write $(X,B+\boldsymbol{\rm M})$. 

Let $(X,B+\boldsymbol{\rm M})/Z$ be a generalized pair and let $f \colon Y \to X$ be a projective birational morphism from a normal variety $Y$. 
Then there is a $\mathbb{Q}$-divisor $B_{Y}$ on $Y$ such that
$$K_{Y}+B_{Y}+\boldsymbol{\rm M}_{Y}=f^{*}(K_{X}+B+\boldsymbol{\rm M}_{X}).$$
We say that a generalized pair $(X,B+\boldsymbol{\rm M})/Z$ is {\em generalized lc} if $\mult_{P}(B_{Y})\leq 1$ for every projective birational morphism $Y \to X$ and every prime divisor $P$ on $Y$. 
\end{defn}

We recall the definition of lc-trivial fibration and the canonical bundle formula. 

\begin{defn}[Lc-trivial fibration for lc pair]\label{defn--lctrivial-fib}
In this note, an {\em lc-trivial fibration}, often denoted by $(X,\Delta) \to Z$, means a projective lc pair $(X,\Delta)$ equipped with a contraction $X \to Z$ of normal projective varieties such that $\Delta$ is a $\mathbb{Q}$-divisor and $K_{X}+\Delta \sim_{\mathbb{Q},Z}0$. 
\end{defn}

In general, lc-trivial fibrations are defined for contractions and sub-pairs that are not necessarily sub-lc (\cite{ambro1}, \cite{ambro}, \cite{fg-bundle}, \cite{fg-lctrivial}, \cite{floris-lazic}, \cite{hu-lctrivial-b-div}). 
However, in this note, we adopt the above convention for simplicity of the definition. 
In particular, we always assume that the lc pairs of lc-trivial fibrations have $\mathbb{Q}$-boundary divisors.

\begin{defn}[Discriminant b-divisor, moduli b-divisor, canonical bundle formula]\label{defn--canbundleformula}
Let $\pi \colon (X,\Delta) \to Z$ be an lc-trivial fibration. 
We define the {\em discriminant $\mathbb{Q}$-b-divisor $\boldsymbol{\rm B}$} and the {\em moduli $\mathbb{Q}$-b-divisor} $\boldsymbol{\rm M}$ on $Z$ as follows: 
Let $\phi \colon Z' \to Z$ be a birational morphism from a normal projective variety $Z'$. 
We first define the discriminant $\mathbb{Q}$-b-divisor $\boldsymbol{\rm B}$. 
Let $X'$ be a resolution of the main component of $X \times_{Z}Z'$, and let $\psi \colon X' \to X$ and  $\pi' \colon X' \to Z'$ be the induced morphisms. 
We define a $\mathbb{Q}$-divisor $\Delta'$ on $X'$ by $K_{X'} +\Delta'=\psi^{*}(K_{X}+\Delta)$. 
For every prime divisor $P$ on $Z'$, let $b_{P}$ be the largest real number such that after we shrink $Z'$ around the generic point $\eta$ of $P$, the subpair $(X',\Delta'+b_P\pi'^{*}P)$ is sub-lc. 
Note that $b_{P}$ is well-defined since $P$ is Cartier at $\eta$. 
Then the trace $\boldsymbol{\rm B}_{Z'}$ of $\boldsymbol{\rm B}$ on $Z'$ is defined by 
\[
\boldsymbol{\rm B}_{Z'}:=\sum_{P}(1-b_{P})P,
\]
where $P$ runs over prime divisors on $Z'$. 
Next, we define the moduli $\mathbb{Q}$-b-divisor $\boldsymbol{\rm M}$. 
We fix a $\mathbb{Q}$-Cartier divisor $L$ on $Z$ such that $K_{X}+\Delta \sim_{\mathbb{Q}}\pi^{*}L$. 
Then the trace $\boldsymbol{\rm M}_{Z'}$ of $\boldsymbol{\rm M}$ on $Z'$ is defined by 
$$\boldsymbol{\rm M}_{Z'}:= \phi^{*} L-(K_{Z'}+\boldsymbol{\rm B}_{Z'}).$$ 
Note that $\boldsymbol{\rm M}_{Z'}$ (and therefore $\boldsymbol{\rm M}$) depends on the choice of $L$. 
We call 
\[
K_X+\Delta\sim_{\mathbb{Q}}\pi^{*}(K_{Z}+\boldsymbol{\rm B}_{Z}+\boldsymbol{\rm M}_{Z})
\]
the {\it canonical bundle formula}. 

Throughout this note, for an lc-trivial fibration $(X,\Delta) \to Z$, unless otherwise stated the discriminant b-divisor (resp.~the moduli b-divisor) is denoted by $\boldsymbol{\rm B}$ (resp.~$\boldsymbol{\rm M}$). 

Let $\boldsymbol{\rm K}$ be the canonical b-divisor on $Z$, that is, the trace of $\boldsymbol{\rm K}$ on a normal variety $Z'$ is the canonical divisor $K_{Z'}$. 
By the results of Kawamata \cite{kawamata-subad-II} and Ambro \cite{ambro1}, there is a projective birational morphism $Z'' \to Z$ such that $\boldsymbol{\rm K}+\boldsymbol{\rm B}$ and $\boldsymbol{\rm M}$ descend to $Z''$ and $\boldsymbol{\rm M}_{Z''}$ is nef. 
In particular, $\boldsymbol{\rm B}$ and $\boldsymbol{\rm M}$ form a generalized lc pair $(Z,\boldsymbol{\rm B}_{Z}+\boldsymbol{\rm M})$.  
\end{defn}

\begin{defn}[Ambro model and log smooth Ambro model]\label{defn--ambromodel}
Let $(X,\Delta) \to Z$ be an lc-trivial fibration. 
We say that a normal variety $Z'$ together with a projective birational morphism $Z' \to Z$ is an {\em Ambro model} if the moduli $\mathbb{Q}$-b-divisor $\boldsymbol{\rm M}$ descends to $Z'$ and $\boldsymbol{\rm M}_{Z'}$ is nef. 
If, additionally, the pair $(Z', \Supp\boldsymbol{\rm B}_{Z'})$ is log smooth, we say that $Z' \to Z$ is a {\em log smooth Ambro model}. 
\end{defn}

In Section \ref{sec4} we study lc-trivial fibrations with log big moduli parts.

\begin{defn}[Log big divisor]
Let $(X,\Delta)$ be a projective sub-lc pair and $M$ an $\mathbb{R}$-Cartier divisor on $X$. 
We say that $M$ is {\em log big} with respect to $(X,\Delta)$ if $M$ is big and for any lc center $S$ of $(X,\Delta)$ with the normalization $S^{\nu} \to S$, the pullback of $M$ to $S^{\nu}$ is big.  
\end{defn}

\begin{defn}[Lc-trivial fibration with log big moduli part]\label{defn--lc-trivial-log-big}
Let $(X,\Delta) \to Z$ be an lc-trivial fibration. 
We say that $(X,\Delta) \to Z$ is an {\em lc-trivial fibration with log big moduli part} if there exists a log smooth Ambro model $Z' \to Z$ such that $\boldsymbol{\rm M}_{Z'}$ is log big with respect to $(Z',\boldsymbol{\rm B}_{Z'})$. 
\end{defn}

Lastly, we define good minimal model and log canonical model.

\begin{defn}[Good minimal model and log canonical model]
Let $X \to Z$ be a projective morphism from a normal variety to a variety and let $(X,\Delta)$ be an lc pair. 
Let $X' \to Z$ be a projective morphism from a normal variety $X'$ and let $\phi \colon X \dashrightarrow X'$ be a birational contraction over $Z$. 
Put $\Delta'=\phi_{*}\Delta$. 
Then $(X',\Delta')$ is called a {\em good minimal model} of $(X,\Delta)$ over $Z$ if 
\begin{itemize}
\item
$K_{X'}+\Delta'$ is $\mathbb{R}$-Cartier and semi-ample over $Z$, and 
\item
for any $\phi$-exceptional prime divisor $E$ on $X$, we have $a(E,X,\Delta)<a(E,X',\Delta')$.
\end{itemize}
We say that $(X',\Delta')$ is a {\em log canonical model} of $(X,\Delta)$ over $Z$ if 
\begin{itemize}
\item
$K_{X'}+\Delta'$ is $\mathbb{R}$-Cartier and ample over $Z$, and 
\item
for any $\phi$-exceptional prime divisor $E$ on $X$, we have $a(E,X,\Delta) \leq a(E,X',\Delta')$.
\end{itemize}
\end{defn}

\section{Proofs of main results}\label{sec3}

In this section, we prove Theorem \ref{thm--lctrivialfibration-boundedfiber} and Theorem \ref{thm--eff-base-point-free}. 

\begin{lem}\label{lem--Cartier-klt}
Let $(X,\Delta)$ be a klt pair and let $D$ be a $\mathbb{Q}$-Cartier divisor on $X$. 
If there is a projective birational morphism $f \colon X' \to X$ from a normal variety $X'$ such that $f^{*}D$ is Cartier, then $D$ is Cartier. 
\end{lem}

\begin{proof}
Because the Cartier property can be checked locally, we may assume that $X$ is affine. 
We may write
$$K_{X'}+\Delta'=f^{*}(K_{X}+\Delta)+E',$$
where $\Delta' \geq 0$ and $E' \geq 0$ have no common components. 
By replacing $f$ with a log resolution of $(X,\Supp \Delta)$, we may assume that $(X',\Delta')$ is a $\mathbb{Q}$-factorial klt pair. 

By running a $(K_{X'}+\Delta')$-MMP over $X$ with scaling of an ample divisor and applying \cite[Theorem 3.5]{birkar-flip} to $(X',\Delta')$ and $f$, we get $\phi\colon X'\dashrightarrow X''$, a finite sequence of steps of a $(K_{X'}+\Delta')$-MMP over $X$, that contracts $E'$. 
Then $\phi_{*}f^{*}D$ is Cartier by the cone and contraction theorem \cite[Theorem 3.7]{kollar-mori}. 
Thus, replacing $X'$ by $X''$, we may assume $E'=0$. 
Then $K_{X'}+\Delta' \sim_{\mathbb{Q},X}0$ and $\Delta'$ is big over $X$. 

By the base point free theorem \cite[Theorem 3.24]{kollar-mori}, for every integer $m\gg0$ it follows that $mf^{*}D$ is linearly equivalent to the pullback of a Cartier divisor on $X$. 
Then $mD$ is Cartier for all $m\gg0$, so $D$ is Cartier.  
\end{proof}

\begin{proof}[Proof of Theorem \ref{thm--lctrivialfibration-boundedfiber}]
Since $mf^{*}(K_{X}+\Delta)$ is Cartier for some resolution $f\colon X' \to X$ of $X$, it follows that the divisor $m\Delta=f_{*}(mf^{*}(K_{X}+\Delta))-mK_{X}$ is a Weil divisor. 
By \cite[Lemma 7.4]{birkar-nefpart}, there exists a positive integer $n'$, depending only on $d$, $m$, and $v$, such that a generalized lc pair $(Z,\boldsymbol{\rm B}_{Z}+\boldsymbol{\rm M})$ defined with the canonical bundle formula satisfies
\begin{itemize}
\item
 $n'(K_{X}+\Delta)\sim n'\pi^{*}(K_{Z}+\boldsymbol{\rm B}_{Z}+\boldsymbol{\rm M}_{Z})$, and
\item
$n' \boldsymbol{\rm M}$ is b-Cartier. 
\end{itemize}
By applying \cite[Theorem 1.7]{birkar-geometry-moduli} to $(F,\Delta|_{F})$ and $A|_{F}$, we may find a positive real number $t_{0}$, depending only on $d$, $m$, and $v$, such that $(F,\Delta|_{F}+t_{0}A|_{F})$ is lc. 
By shrinking $U$, we may assume that $\bigl(\pi^{-1}(U),(\Delta+t_{0}A)|_{\pi^{-1}(U)}\bigr)$ is lc. 
By shrinking $U$ again in order to remove the vertical part of $A$, we may assume that all components of $A|_{\pi^{-1}(U)}$ dominate $U$. 
Then we can remove the vertical part of $A$ with respect to $X \to Z$ without loss of generality, and therefore we may assume that all components of $A$ dominate $Z$. 

We put $D=K_{Z}+\boldsymbol{\rm B}_{Z}+\boldsymbol{\rm M}_{Z}$. 
By Lemma \ref{lem--Cartier-klt}, it is sufficient to prove the existence of a positive integer $n''$ depending on $d$, $m$, $v$, $n'$, and $t_{0}$ such that $n''\phi^{*}D$ is Cartier for some resolution $\phi \colon Z' \to Z$ of $Z$. 
Indeed, supposing the existence of such $n''$, then the positive integer $n:=mn'n''$ depends only on $d$, $m$, and $v$, and it follows that $n(K_{X}+\Delta)\sim n\pi^{*}(K_{Z}+\boldsymbol{\rm B}_{Z}+\boldsymbol{\rm M}_{Z})$, $n \boldsymbol{\rm M}$ is b-Cartier, and $n\phi^{*}(K_{Z}+\boldsymbol{\rm B}_{Z}+\boldsymbol{\rm M}_{Z})$ is Cartier for some resolution $\phi \colon Z' \to Z$ of $Z$. 
These facts show that this $n$ is the desired positive integer. 
From now on, we will prove the existence of such $n''$ as above. 
We will follow \cite[Proof of Lemma 7.3]{birkar-nefpart}. 

\begin{step2}\label{step1--lem--lctrivialfibration-boundedfiber}
In this step we construct a contraction $X' \to Z'$ such that $X'$ and $Z'$ are birational to $X$ and $Z$ respectively, and we define $\mathbb{Q}$-divisors $\Delta'$ and $A'$ on $X'$. 

Let $h\colon Z' \to Z$ be a log resolution of $(Z, \Supp D)$. Let $\Sigma'$ be a divisor on $Z'$ whose support contains $\Supp h^{*}D \cup {\rm Ex}(h)$. 
Shrinking $U$ and adding some prime divisors to $\Sigma'$ if necessary, we may assume that $Z'\setminus \Sigma' = h^{-1}(U)$ and the image of the vertical part of $\Delta$ maps into $Z \setminus U$. 
Then, $h$ is an isomorphism over $U$. 
Let $f \colon X' \to X$ be a resolution of $X$ such that $mf^{*}(K_{X}+\Delta)$ is Cartier. 
Replacing $f$ by an other resolution  of $X$ if necessary, we may assume that the induced map $\pi' \colon X' \dashrightarrow Z'$ is a morphism and $\bigl(X', \Supp\,(f^{-1}_{*}(\Delta+A)+\pi'^{*}\Sigma')\cup {\rm Ex}(f)\bigr)$ is log smooth. 
We define a $\mathbb{Q}$-divisor $\Delta'$ on $X'$ as follows: For a prime divisor $P$ on $X'$, we define 
$$
\mult_{P}(\Delta'):=\left\{
\begin{array}{lll}1&&\text{($P$ is $f$-exceptional or $\pi'(P)\subset \Sigma'$)}\\ \mult_{P}(f^{-1}_{*}\Delta)&&\text{(otherwise).}
\end{array}\right.
$$
We put $V=\pi^{-1}(U)$, $V'=\pi'^{-1}(Z'\setminus \Sigma')$, and $f_{V'}=f|_{V'}$. 
Since $Z'\setminus \Sigma' = h^{-1}(U)$, we have $V'=f^{-1}(V)$. 
We have the following diagrams.
 $$
\xymatrix{
X \ar@{->}[d]_{\pi}&X' \ar[l]_{f}\ar@{->}[d]^{\pi'}\\
Z&\ar[l]^{h}Z'
}
\qquad\qquad
\xymatrix{
V \ar[d]&V'\ar[l]_{f_{V'}}\ar[d]\\
U&Z'\setminus \Sigma'\ar[l]_(0.58){\simeq}
}
$$
We put $A'=f^{-1}_{*}A$. 
Since all components of $A$ dominate $Z$ and $\bigl(V, (\Delta+t_{0}A)|_{V}\bigr)$ is lc, $\Delta'+t_{0}A'$ is a boundary $\mathbb{R}$-divisor. 
Then $(X',\Delta'+t_{0}A')$ is lc. 
We may write
\begin{equation*}
\begin{split}
K_{X'}+\Delta'=f^{*}(K_{X}+\Delta)+E'+\Xi',
\end{split}
\end{equation*}
where $E'$ is an effective $f$-exceptional $\mathbb{Q}$-divisor whose components intersect $V'$ and $\Xi'$ is a $\mathbb{Q}$-divisor whose support is mapped into $\Sigma'$ by $\pi'$. 
We may write 
$$f_{V'}^{*}(A|_{V})=A'|_{V'}+E_{0}$$
for some $f_{V'}$-exceptional $\mathbb{Q}$-divisor $E_{0} \geq 0$ on $V'$. 
By the above relations we obtain
\begin{equation*}\tag{$1$}\label{proof-lem--lctrivialfibration-boundedfiber-(1)}
(K_{X'}+\Delta'+t_{0}A')|_{V'}=f_{V'}^{*}\bigl((K_{X}+\Delta+t_{0}A)|_{V}\bigr)+E'|_{V'}-t_{0}E_{0}
\end{equation*}
and $E'|_{V'}-t_{0} E_{0}$ is $f_{V'}$-exceptional. 
For any $f_{V'}$-exceptional prime divisor $P'$ on $V'$, the definition of $\Delta'$ implies $\mult_{P'}(\Delta'|_{V'})=1$. 
Since $\Delta'+t_{0}A'$ is a boundary $\mathbb{R}$-divisor, we have $\mult_{P'}((\Delta'+t_{0}A')|_{V'})=1$. 
Since the pair $\bigl(V, (\Delta+t_{0}A)|_{V}\bigr)$ is lc, we have 
$$\mult_{P'}(E'|_{V'}- t_{0}E_{0})=a\bigl(P',V,(\Delta+t_{0}A)|_{V}\bigr)\geq 0.$$ 
This implies that $E'|_{V'}-t_{0} E_{0}$ is effective. 
\end{step2}

\begin{step2}\label{step2--lem--lctrivialfibration-boundedfiber}
In this step we prove that $(X',\Delta'+t_{0}A')$ has a good minimal model over $Z'$. 

Since $\bigl(V,(\Delta+t_{0}A)|_{V}\bigr)$ is an lc pair and $K_{V}+\Delta|_{V}+t_{0}A|_{V} \sim_{\mathbb{Q},U}t_{0}A|_{V}$ is ample over $U$, the relation (\ref{proof-lem--lctrivialfibration-boundedfiber-(1)}) in Step \ref{step1--lem--lctrivialfibration-boundedfiber} and \cite[Lemma 2.15]{has-mmp} show that $\bigl(V', (\Delta'+t_{0}A')|_{V'}\bigr)$ has a good minimal model over $Z'\setminus \Sigma'$. 
By construction of $\Delta'$, we can find an effective $\mathbb{Q}$-divisor $T'$ on $X'$, which is a multiple of $\pi'^{*}\Sigma'$ by a positive rational number, such that $\Delta'-T'\geq 0$. 
Then the pair $(X',\Delta'-T'+t_{0}A')$ is lc and all lc centers of the pair intersect  $\pi'^{-1}(Z'\setminus \Sigma')$ (cf.~\cite[Lemma 2.27]{kollar-mori}). 
By \cite[Theorem 1.2]{has-mmp}, the lc pair $(X',\Delta'-T'+t_{0}A')$ has a good minimal model over $Z'$. 
Moreover, the relation $T'\sim_{\mathbb{Q},Z'}0$ shows 
$$K_{X'}+\Delta'-T'+t_{0}A'\sim_{\mathbb{Q},Z'}K_{X'}+\Delta'+t_{0}A'.$$
From these facts, it follows that $(X',\Delta'+t_{0}A')$ has a good minimal model over $Z'$ (see, for example, \cite[Proof of Lemma 3.6.9]{bchm}). 

\end{step2}

\begin{step2}\label{step3--lem--lctrivialfibration-boundedfiber}
In this step we define some varieties and $\mathbb{R}$-divisors. 

By running a $(K_{X'}+\Delta'+t_{0}A')$-MMP over $Z'$, we get a birational contraction
$$(X',\Delta'+t_{0}A') \dashrightarrow (X'',\Delta''+t_{0}A'')$$
over $Z'$ to a good minimal model $(X'',\Delta''+t_{0}A'')$. 
Let $g \colon X'' \to Y$ be the contraction over $Z'$ induced by $K_{X''}+\Delta''+t_{0}A''$. 
By the relation (\ref{proof-lem--lctrivialfibration-boundedfiber-(1)}) in Step \ref{step1--lem--lctrivialfibration-boundedfiber} and the ampleness of $(K_{X}+\Delta+t_{0}A)|_{V}\sim_{\mathbb{Q},U}t_{0}A|_{V}$ over $U$, we see that $K_{X''}+\Delta''+t_{0}A''$ is big over $Z'$. 
Hence $g$ is birational. 
We put $\Gamma=g_{*}(\Delta''+t_{0}A'')$ and we denote $Y \to Z'$ by $\pi_{Y}$. 
We have the following diagram
 $$
\xymatrix{
X \ar@{->}[d]_{\pi}&X' \ar[l]_{f}\ar@{-->}[r]\ar@{->}[rd]_(0.35){\pi'}& X'' \ar@{->}[d]\ar[r]^{g}& Y \ar[dl]^(0.35){\pi_{Y}}\\
Z&&Z'\ar[ll]^{h}
}
$$
such that $K_{Y}+\Gamma$ is ample over $Z'$ and 
$K_{X''}+\Delta''+t_{0}A''=g^{*}(K_{Y}+\Gamma)$. 
Let $G$ be a sufficiently general fiber of $\pi_{Y}$. 
By construction, we have
$${\rm vol}\bigl((K_{Y}+\Gamma)|_{G}\bigr)={\rm vol}\bigl((K_{X}+\Delta+t_{0}A)|_{F}\bigr)=t_{0}^{\dim F}{\rm vol}(A|_{F}).$$
Because ${\rm vol}(A|_{F}) = v$ and $t_{0} \leq 1$, which follows from the fact that $A$ is a Weil divisor, we see that ${\rm vol}\bigl((K_{Y}+\Gamma)|_{G}\bigr)\leq v$. 
\end{step2}

\begin{step2}\label{step4--lem--lctrivialfibration-boundedfiber}
In this step we prove that for the generic point $\eta$ of any irreducible component of $\Sigma'$, the multiplicity of the fiber of $\pi_{Y}$ over $\eta$ has an upper bound depending only on $d$, $m$, $t_{0}$, and $v$. 

We may assume that $Z'$ is a curve by cutting with hyperplane sections. 
We fix the generic point $\eta$ of any irreducible component $\Sigma'$. 
Then we may write
$$\pi_{Y}^{*}\eta=\sum_{i} \mu_{i}G_{i},$$ where $G_{i}$ are prime divisors on $Y$. 
For each $i$, let $G_{i}^{\nu}$ be the normalization of $G_{i}$. 
For every $i$, the inequality ${\rm vol}\bigl((K_{Y}+\Gamma)|_{G}\bigr)\leq v$ and the computation of the volumes (see, for example, \cite[Proposition 1.35]{kollar-mori}) imply
\begin{equation*}
\begin{split}
 v \geq G \cdot (K_{Y}+\Gamma)^{\dim G}\geq \mu_{i}G_{i}\cdot (K_{Y}+\Gamma)^{\dim G_{i}} =& \mu_{i}\cdot ((K_{Y}+\Gamma)|_{G_{i}^{\nu}})^{\dim G_{i}^{\nu}}
\\
=&\mu_{i}\cdot{\rm vol}\bigl((K_{Y}+\Gamma)|_{G_{i}^{\nu}}\bigr)>0. 
\end{split}
\end{equation*}
By the definition of $\Delta'$ in Step \ref{step1--lem--lctrivialfibration-boundedfiber},  $\lfloor \Delta' \rfloor$ contains all components of $\pi'^{*}\Sigma'$.  
Hence $G_{i}$ is a component of $\lfloor \Gamma \rfloor$. 
Since $m\Delta$ is a Weil divisor, which follows from the hypothesis of Theorem \ref{thm--lctrivialfibration-boundedfiber}, the definition of $\Delta'$ in Step \ref{step1--lem--lctrivialfibration-boundedfiber} shows that $m \Delta'$ is a Weil divisor on $X'$. 
Hence, the coefficients of $\Gamma$ belong to $\frac{1}{m}\mathbb{Z}_{\geq 0}\cup t_{0}\mathbb{Z}_{\geq 0}$. 
By applying divisorial adjunction to $(Y,\Gamma)$ and $G_{i}$ and applying the DCC for volumes \cite[Theorem 1.3]{hmx-acc}, the volume 
${\rm vol}\bigl((K_{Y}+\Gamma)|_{G_{i}^{\nu}}\bigr)$ 
is bounded from below by a positive real number depending only on $d$, $m$, and $t_{0}$. 
Thus, $\mu_{i}$ has an upper bound depending only on $d$, $m$, $t_{0}$, and $v$. 
This shows that for the generic point $\eta$ of any irreducible component of $\Sigma'$, the multiplicity of the fiber of $\pi_{Y}$ over $\eta$ has an upper bound depending only on $d$, $m$, $t_{0}$, and $v$. 
\end{step2}

\begin{step2}\label{step5--lem--lctrivialfibration-boundedfiber}
In this step we prove that $n''h^{*}D$ is Cartier for some positive integer $n''$ that depends only on $d$, $m$, $t_{0}$, $v$, and $n'$, where $n'$ is the positive integer defined at the start of this proof. 

We define a $\mathbb{Q}$-divisor $\Theta'$ on $X'$ by $K_{X'}+\Theta'=f^{*}(K_{X}+\Delta)$. 
Let $\Theta_{Y}$ be the birational transform of $\Theta'$ on $Y$. 
By the hypothesis of Theorem \ref{thm--lctrivialfibration-boundedfiber} that $mf^{*}(K_{X}+\Delta)$ is Cartier, the birational transform $m(K_{Y}+\Theta_{Y})$ is a Weil divisor on $Y$. 
Then there is a rational function $\sigma$ on $Y$ such that
\begin{equation*}\tag{$2$}\label{proof-lem--lctrivialfibration-boundedfiber-(2)}
mn'(K_{Y}+\Theta_{Y})+{\rm div}(\sigma) =mn' \pi_{Y}^{*}h^{*}D
\end{equation*}
as $\mathbb{Q}$-divisors. 
Moreover, the left hand side is a Weil divisor.  
We can write 
$$h^{*}D=\sum_{j}a_{j}D_{j},$$
where $D_{j}$ are prime divisors on $Z'$. 
By Step \ref{step4--lem--lctrivialfibration-boundedfiber}, the multiplicity of the fiber of $\pi_{Y}$ over the generic point $\eta$ of any irreducible component of $\Sigma'$ has an upper bound, which we denote by $\beta$, depending only on $d$, $m$, $t_{0}$, and $v$. 
Therefore, $\pi_{Y}^{*}D_{j}$ has a component $Q_{j}$ such that $\pi_{Y}(Q_{j})=D_{j}$ and $\mult_{Q_{j}}(\pi_{Y}^{*}D_{j})\leq \beta$. 
On the other hand, the relation (\ref{proof-lem--lctrivialfibration-boundedfiber-(2)}) implies $mn'a_{j}\cdot \mult_{Q_{j}}(\pi_{Y}^{*}D_{j}) \in \mathbb{Z}$. 
From these facts, we have $mn' \lfloor \beta \rfloor !a_{j} \in \mathbb{Z}$. 

We define $n'':=mn' \lfloor \beta \rfloor !$. 
Then $n''$ depends only on $d$, $m$, $t_{0}$,  $v$, and $n'$. 
Moreover, $n''h^{*}D$ is a Weil divisor. 
In particular, $n''h^{*}D$ is Cartier. 
\end{step2}

By defining $n:=mn'n''$ and $\phi:=h\colon Z' \to Z$, we complete the proof. 
If there is a klt pair structure on $Z$, then $n''D$ is Cartier by Lemma \ref{lem--Cartier-klt}. 
\end{proof}

\begin{proof}[Proof of Theorem \ref{thm--eff-base-point-free}]
Let $\pi \colon (X,\Delta) \to Z$ be the contraction as in Theorem \ref{thm--eff-base-point-free}. 
Since $A$ is big over $Z$, we can find an effective $\mathbb{Q}$-divisor $B \sim_{\mathbb{Q},Z}A$. 
We pick $t \in \mathbb{Q}_{>0}$ such that $(X,\Delta+tB)$ is klt. 
Since $B$ is big over $Z$, it follows from \cite{bchm} that $(X,\Delta+tB)$ has a good minimal model over $Z$. 
In particular, $(X,\Delta+tB)$ has a log canonical model
$$(X,\Delta+tB)\dashrightarrow (X',\Delta'+tB')$$
over $Z$. 
We take a common resolution $g\colon W \to X$ and $g' \colon W \to X'$ of $X\dashrightarrow X'$. 
Since $K_{X}+\Delta\sim_{\mathbb{Q},Z}0$, by the negativity lemma, we have
$g^{*}(K_{X}+\Delta)=g'^{*}(K_{X'}+\Delta')$. 
Thus, $e(K_{X'}+\Delta')$ is semi-ample and $X' \to Z$ is the contraction induced by $e(K_{X'}+\Delta')$. 
Furthermore, since $m(K_{X}+\Delta)$ is Cartier, $mg'^{*}(K_{X'}+\Delta')$ is Cartier, so  $m(K_{X'}+\Delta')$ is Cartier by Lemma \ref{lem--Cartier-klt}. 
Let $A'$ be the birational transform of $A$ on $X'$. 
Then 
$$A'\sim_{\mathbb{Q},Z}B'\sim_{\mathbb{Q},Z}\frac{1}{t}(K_{X'}+\Delta'+tB')$$
because $A\sim_{\mathbb{Q},Z}B\sim_{\mathbb{Q},Z}\frac{1}{t}(K_{X}+\Delta+tB)$. 
Thus, $A'$ is ample over $Z$.  
By these relations and the negativity lemma, we can write $g^{*}A=g'^{*}A'+E$ with an effective $g'$-exceptional $\mathbb{Q}$-divisor $E$ on $W$. 
From this fact, we have ${\rm vol}(A'|_{F'}) = v$, where $F'$ is a sufficiently general fiber of $X' \to Z$. 
Now it is easy to see that Theorem \ref{thm--eff-base-point-free} holds for $(X,\Delta) \to Z$ and $A$ if and only if Theorem \ref{thm--eff-base-point-free} holds for $(X',\Delta') \to Z$ and $A'$.  
Thus, replacing $(X,\Delta)$ and $A$ by $(X',\Delta')$ and $A'$ respectively, we may assume that $A$ is ample over $Z$. 

Since $m(K_{X}+\Delta)$ is Cartier, the coefficients of $\Delta$ belong to $\frac{1}{m}\mathbb{Z}_{>0}\cap[0,1]$, which is a finite set of rational numbers. 
By applying \cite[Corollary 1.4]{birkar-geometry-moduli} to general fibers of $\pi$, we can find a positive integer $I$, depending only on $d$ and $m$, such that we have 
$G \sim IA$ for some Weil divisor $G$ whose horizontal part, which we denote by $L$, is effective. 
Note that $L$ is not necessarily $\mathbb{Q}$-Cartier, but $L$ is a Weil divisor, and $L$ is linearly equivalent to the $\mathbb{Q}$-Cartier divisor $A$ over some open subset $U \subset Z$, hence $L|_{\pi^{-1}(U)}$ is $\mathbb{Q}$-Cartier. 
Then 
$${\rm vol}(L|_{F})={\rm vol}(IA|_{F})=I^{d-\dim Z}v,$$
where $F$ is a sufficiently general fiber of $\pi$. 

We apply Theorem \ref{thm--lctrivialfibration-boundedfiber} to $\pi \colon (X,\Delta) \to Z$ and $L$. 
Because there is a klt pair structure on $Z$ (\cite[Theorem 4.1]{ambro}), there exists $n' \in \mathbb{Z}_{>0}$, depending only on $d$, $m$, and $v$, such that $n'(K_{X}+\Delta)\sim \pi^{*}D$ for some Cartier divisor $D$ on $Z$. 
By  \cite[Theorem 4.1]{ambro} there is a klt pair $(Z,B)$ such that $K_{Z}+B\sim_{\mathbb{Q}}\frac{1}{n'}D$. 
Now $eD$ is ample by construction of $\pi$. 
By \cite[Theorem 1.1]{kollar-eff-basepoint-free} and \cite[Lemma 7.1]{fujino-eff-slc},  there exists $n'' \in \mathbb{Z}_{>0}$, depending only on $\dim Z$, such that $2n''eD$ is very ample. 
We denote $n''$ by $n''(\dim Z)$. 
Since $\dim Z$ satisfies $0 \leq \dim Z\leq d$, the positive integer $n:=2n'\prod_{i=0}^{d}n''(i)$ satisfies the condition of Theorem \ref{thm--eff-base-point-free}. 
\end{proof}

\section{Lc-trivial fibration with log big moduli part}\label{sec4}

In this section, we study lc-trivial fibrations with log big moduli parts.

\begin{thm}\label{thm--lctrivialfibration-boundedfiber-abundance}
For every $d$, $m \in \mathbb{Z}_{>0}$, and $v \in \mathbb{R}_{>0}$, there exists $n \in \mathbb{Z}_{>0}$, depending only on $d$, $m$, and $v$, satisfying the following.  
Let $\pi \colon (X,\Delta) \to Z$ be an lc-trivial fibration with log big moduli part, let $F$ be the sufficiently general fiber of $\pi$, and let $A \geq 0$ be a Weil divisor on $X$ such that 
\begin{itemize}
\item
$\dim X=d$, 
\item
$m(K_{X}+\Delta)$ is nef and Cartier, 
\item
$A|_{\pi^{-1}(U)}$ is $\mathbb{Q}$-Cartier and ample over $U$ for some open subset $U \subset Z$,
\item
$(F,\Delta|_{F}+tA|_{F})$ is an lc pair for some real number $t>0$, and 
\item
${\rm vol}(A|_{F}) = v$. 
\end{itemize}
Then $n(K_{X}+\Delta)$ is base point free. 
\end{thm}

\begin{proof}
By Theorem \ref{thm--lctrivialfibration-boundedfiber} there is $n'$ that depends on $d$, $m$, and $v$ such that a generalized lc pair $(Z,\boldsymbol{\rm B}_{Z}+\boldsymbol{\rm M})$ defined by the canonical bundle formula satisfies
\begin{itemize}
\item
 $n'(K_{X}+\Delta)\sim n'\pi^{*}(K_{Z}+\boldsymbol{\rm B}_{Z}+\boldsymbol{\rm M}_{Z})$, 
\item
$n' \boldsymbol{\rm M}$ is b-Cartier, and
\item
$n'\phi^{*}(K_{Z}+\boldsymbol{\rm B}_{Z}+\boldsymbol{\rm M}_{Z})$ is Cartier for some resolution $\phi \colon Z' \to Z$ of $Z$.   
\end{itemize}
Put $D=(K_{Z}+\boldsymbol{\rm B}_{Z}+\boldsymbol{\rm M}_{Z})$. 
By Lemma \ref{lem--Cartier-klt}, for any resolution $\psi \colon Z'' \to Z$, the Cartier property of $n'\phi^{*}D$ implies that $n'\psi^{*}D$ is Cartier. 
Since $(X,\Delta) \to Z$ is an lc-trivial fibration with log big moduli part, replacing $Z'$ by an appropriate log smooth Ambro model, we may assume that $(Z',\Supp \boldsymbol{\rm B}_{Z'})$ is log smooth and $\boldsymbol{\rm M}_{Z'}$ is log big with respect to $(Z',\boldsymbol{\rm B}_{Z'})$. 
Then 
$n'(K_{Z'}+\boldsymbol{\rm B}_{Z'}+\boldsymbol{\rm M}_{Z'})=n'\phi^{*}D$
 is Cartier. 

We can write $\boldsymbol{\rm B}_{Z'}=B'-E'$, where $B' \geq 0$ and $E' \geq 0$ have no common components. 
Then 
$$K_{Z'}+B'+\boldsymbol{\rm M}_{Z'}=\phi^{*}(K_{Z}+\boldsymbol{\rm B}_{Z}+\boldsymbol{\rm M}_{Z})+E'.$$ 
By running a $(K_{Z'}+B'+\boldsymbol{\rm M}_{Z'})$-MMP over $Z$, we get a birational contraction 
$$(Z',B'+\boldsymbol{\rm M}) \dashrightarrow (W,B'_{W}+\boldsymbol{\rm M})$$
over $Z$ such that $K_{W}+B'_{W}+\boldsymbol{\rm M}$ is the limit of movable divisors over $Z$. 
By applying \cite[Lemma 3.3]{birkar-flip} to the birational transform of $E'$, we see that the birational map $Z' \dashrightarrow W$ contracts $E'$. 
Then $B'_{W}=\boldsymbol{\rm B}_{W}$. 
Moreover, the divisor $n'(K_{Z'}+\boldsymbol{\rm B}_{Z'}+\boldsymbol{\rm M}_{Z'})$ is trivial over the extremal contractions of the MMP over $Z$. 
Pick $t\in (0,1]$ such that $Z' \dashrightarrow W$ is a sequence of steps of a $(K_{Z'}+(1-t)(B'+\boldsymbol{\rm M}_{Z'}))$-MMP. 
Since $(Z',B')$ is a $\mathbb{Q}$-factorial dlt pair and $\boldsymbol{\rm M}_{Z'}$ is nef and big, we can find a klt pair $(Z',\Gamma')$ such that $K_{Z'}+\Gamma'\sim_{\mathbb{R}}K_{Z'}+(1-t)(B'+\boldsymbol{\rm M}_{Z'})$. 
Then $Z' \dashrightarrow W$ is a sequence of steps of a $(K_{Z'}+\Gamma')$-MMP. 
By the cone and contraction theorem \cite[Theorem 3.7]{kollar-mori}, the divisor $n'(K_{W}+\boldsymbol{\rm B}_{W}+\boldsymbol{\rm M}_{W})$ is Cartier and it is the pullback of $K_{Z}+\boldsymbol{\rm B}_{Z}+\boldsymbol{\rm M}_{Z}$ to $W$. 
By \cite[Theorem 5.8]{has-iitakafibration}, there exists $n''$, depending only on $\dim Z$ and $n'$, such that $n''(K_{W}+\boldsymbol{\rm B}_{W}+\boldsymbol{\rm M}_{W})$ is base point free. 
We denote $n''$ by $n''(\dim Z,\,n')$, and we define 
$$n:=mn'\prod_{i=0}^{d}n''(i,\,n').$$ 
Let $f \colon X' \to X$ be a resolution of $X$ such that the induced map $\pi' \colon X' \dashrightarrow W$ is a morphism. 
Then $\pi'$ is a contraction and
$$f^{*}\bigl(n(K_{X}+\Delta)\bigr) \sim \pi'^{*}\bigl(n(K_{W}+\boldsymbol{\rm B}_{W}+\boldsymbol{\rm M}_{W})\bigr).$$
Since $n(K_{X}+\Delta)$ is Cartier, we see that $n(K_{X}+\Delta)$ is base point free. 
\end{proof}

\begin{thm}\label{thm--lctrivialfibration-boundedfiber-antiabundance}
For every $d$, $m \in \mathbb{Z}_{>0}$, and $v \in \mathbb{R}_{>0}$, there exists $n \in \mathbb{Z}_{>0}$, depending only on $d$, $m$, and $v$, satisfying the following.  
Let $\pi \colon (X,\Delta) \to Z$ be an lc-trivial fibration with log big moduli part, let $F$ be the sufficiently general fiber of $\pi$, and let $A \geq 0$ be a Weil divisor on $X$ such that 
\begin{itemize}
\item
$\dim X=d$, 
\item
$-m(K_{X}+\Delta)$ is nef and Cartier, 
\item
$A|_{\pi^{-1}(U)}$ is $\mathbb{Q}$-Cartier and ample over $U$ for some open subset $U \subset Z$,
\item
$(F,\Delta|_{F}+tA|_{F})$ is an lc pair for some real number $t>0$, and 
\item
${\rm vol}(A|_{F}) = v$. 
\end{itemize}
Then $-n(K_{X}+\Delta)$ is base point free. 
In particular, $-(K_{X}+\Delta)$ is semi-ample. 
\end{thm}

\begin{proof}
By Theorem \ref{thm--lctrivialfibration-boundedfiber} there is $n'$ that depends on $d$, $m$, and $v$ such that a generalized lc pair $(Z,\boldsymbol{\rm B}_{Z}+\boldsymbol{\rm M})$ defined by the canonical bundle formula satisfies
\begin{itemize}
\item
 $n'(K_{X}+\Delta)\sim n'\pi^{*}(K_{Z}+\boldsymbol{\rm B}_{Z}+\boldsymbol{\rm M}_{Z})$, 
\item
$n' \boldsymbol{\rm M}$ is b-Cartier, and
\item
$n'\phi^{*}(K_{Z}+\boldsymbol{\rm B}_{Z}+\boldsymbol{\rm M}_{Z})$ is Cartier for some resolution $\phi \colon Z' \to Z$ of $Z$.   
\end{itemize}
Put $D=(K_{Z}+\boldsymbol{\rm B}_{Z}+\boldsymbol{\rm M}_{Z})$. 
As in the proof of Theorem \ref{thm--lctrivialfibration-boundedfiber-abundance}, replacing $Z'$ by an appropriate log smooth Ambro model, we may assume that $(Z',\Supp \boldsymbol{\rm B}_{Z'})$ is log smooth and $\boldsymbol{\rm M}_{Z'}$ is log big with respect to $(Z',\boldsymbol{\rm B}_{Z'})$. 
Then the divisor
$$n'(K_{Z'}+\boldsymbol{\rm B}_{Z'}+\boldsymbol{\rm M}_{Z'}-2\phi^{*}D)=-n'\phi^{*}D$$
 is Cartier. 
Then the proof of Theorem \ref{thm--lctrivialfibration-boundedfiber-abundance} works by replacing $\boldsymbol{\rm M}$ with $\boldsymbol{\rm M}-2\overline{D}$, where $\overline{D}$ is the Cartier closure of $D$ as a b-divisor. 
\end{proof}

Finally, we prove the minimal model theory for lc pairs admiting an lc-trivial fibration with log big moduli part. 
Note that we do not need Theorem \ref{thm--lctrivialfibration-boundedfiber} for the proof.

\begin{thm}\label{thm--minimalmodeltheory-logbigmoduli}
Let $(X,\Delta)$ be a projective lc pair admitting an lc-trivial fibration with log big moduli part. 
Then $(X,\Delta)$ has a good minimal model or a Mori fiber space. 
\end{thm}

\begin{proof}
If $K_{X}+\Delta$ is not pseudo-effective, then $(X,\Delta)$ has a Mori fiber space by \cite[Theorem 1.7]{hashizumehu}. 
Hence we may assume that $K_{X}+\Delta$ is pseudo-effective. 
Let $(X,\Delta) \to Z$ be an lc-trivial fibration with log big moduli part, and let $\phi \colon Z' \to Z$ be a log smooth Ambro model such that $\boldsymbol{\rm M}_{Z'}$ is log big with respect to $(Z',\boldsymbol{\rm B}_{Z'})$. 
We may write
$$K_{Z'}+B'+\boldsymbol{\rm M}_{Z'}=\phi^{*}(K_{Z}+\boldsymbol{\rm B}_{Z}+\boldsymbol{\rm M}_{Z})+E'$$
where $B' \geq 0$ and $E' \geq 0$ have no common components. 
By \cite[Theorem 1.1]{has-iitakafibration}, the generalized lc pair $(Z',B'+\boldsymbol{\rm M})$ has a good minimal model. 
By \cite[Section 4]{bhzariski}, $K_{Z}+\boldsymbol{\rm B}_{Z}+\boldsymbol{\rm M}_{Z}$ birationally has a Nakayama--Zariski decomposition (for definition of the Nakayama--Zariski decomposition, see \cite[III, 1.12. Definition]{nakayama}) with semi-ample positive part. 
Let $g \colon Z'' \to Z$ be a resolution of $Z$ such that  $P_{\sigma}(g^{*}(K_{Z}+\boldsymbol{\rm B}_{Z}+\boldsymbol{\rm M}_{Z}))$, which is the positive part of the Nakayama--Zariski decomposition of $g^{*}(K_{Z}+\boldsymbol{\rm B}_{Z}+\boldsymbol{\rm M}_{Z})$, is semi-ample. 
We take a resolution $h \colon X' \to X$ of $X$ such that the induced map $\psi \colon X' \dashrightarrow Z''$ is a morphism.  
 $$
\xymatrix{
X \ar[d]&X' \ar[l]_{h}\ar[d]^{\psi}\\
Z&Z''\ar[l]^{g}
}
$$
Then $h^{*}(K_{X}+\Delta)\sim_{\mathbb{Q}}\psi^{*}g^{*}(K_{Z}+\boldsymbol{\rm B}_{Z}+\boldsymbol{\rm M}_{Z})$. 
By applying \cite[III, 5.17 Corollary]{nakayama} to $\psi$, we obtain the following relation between the positive parts of the Nakayama--Zariski decompositions 
$$P_{\sigma}(h^{*}(K_{X}+\Delta)) \sim_{\mathbb{Q}}\psi^{*}P_{\sigma}(g^{*}(K_{Z}+\boldsymbol{\rm B}_{Z}+\boldsymbol{\rm M}_{Z})).$$
Therefore the positive part of the Nakayama--Zariski decomposition of $h^{*}(K_{X}+\Delta)$ is semi-ample.  
By applying \cite[Theorem 2.23]{has-finite} (see also \cite[Theorem 1.1]{bhzariski}) and \cite[Theorem 1.7]{hashizumehu} to $(X,\Delta)$, the existence of a good minimal model of $(X,\Delta)$ follows.  
\end{proof}


\end{document}